\documentclass[a4paper,
    11pt,
final,
]{scrartcl}

\usepackage{import}

\title{Computation of Generalized Derivatives for Abs-Smooth Functions by Backward
    Mode Algorithmic Differentiation and Implications to Deep Learning
}

\author{Lukas Baumgärtner\,\orcidlink{0000-0003-1007-4815} \and Franz Bethke\,\orcidlink{0000-0002-0553-4242} }

\RequirePackage{import}

\usepackage{ifluatex}
\usepackage{ifdraft}
\usepackage{etoolbox}

\ifluatex \usepackage{fontspec}
\else \usepackage[utf8]{inputenc}
  \usepackage[T1]{fontenc}
\fi

\usepackage[english]{babel}

\usepackage[useregional]{datetime2}

\usepackage[
    babel,
    tracking=true,
    final]
{microtype}

\usepackage{amsmath}
\usepackage{amssymb}
\usepackage{amsfonts}
\usepackage{amsthm}

\usepackage{mathtools}
\mathtoolsset{showonlyrefs}

\usepackage{xcolor}

\usepackage{booktabs}
\usepackage{multirow}

\usepackage{caption}
\usepackage{subcaption}

\usepackage{algorithm}
\usepackage[noend]{algpseudocode}

\makeatletter
\@ifclassloaded{beamer}
{\usepackage[export]{adjustbox}
}{
    \usepackage[final]{graphicx}
}
\makeatother
\usepackage[noend]{algpseudocode}

\makeatletter
\@ifclassloaded{beamer}
{\usepackage[export]{adjustbox}
}{
    \usepackage[final]{graphicx}
}
\makeatother

\usepackage{pgfplots}
\pgfplotsset{compat=1.18}
\usepgfplotslibrary{external}
\pgfkeys{/pgf/images/include external/.code=\includegraphics{sensibles/./#1}}

\newcommand{\externalizeifable}{\makeatletter
    \ifnum\pdf@shellescape=1\relax \tikzexternalize[prefix=externalized/]\else \relax \fi \makeatother
}

\usepackage{fvextra}  \usepackage{csquotes}
\usepackage[
    backend=biber,
    style=alphabetic,
    doi=false,
    url=true,
    giveninits=true,
    isbn=true,
]{biblatex}

\makeatletter
\@ifclassloaded{beamer}
{}{
    \usepackage[
        hidelinks,
        draft=false
    ]{hyperref}
}
\makeatother

\usepackage{orcidlink}
 \RequirePackage{mathtools}
\RequirePackage{amssymb}

\newcommand{\grad}{\nabla}
\newcommand{\jac}{\mathrm{D}}

\newcommand{\clarke}[1]{\partial_C#1}
\newcommand{\bouligand}[1]{\partial_B#1}
\NewDocumentCommand{\goldstein}{ O{\varepsilon} m }{\partial_{#1}#2}
\renewcommand{\d}[1]{\,\mathrm{d}#1}
\newcommand{\dd}[1]{\frac{\d{}}{\d{#1}}\,}

\newcommand{\ols}[1]{\mskip.5\thinmuskip\overline{\mskip-.5\thinmuskip {#1} \mskip-.5\thinmuskip}\mskip.5\thinmuskip} \newcommand{\olsi}[1]{\,\overline{\!{#1}}} \makeatletter
\newcommand\closure[1]{
  \tctestifnum{\count@stringtoks{#1}>1} {\ols{#1}} {\olsi{#1}} }
\long\def\count@stringtoks#1{\tc@earg\count@toks{\string#1}}
\long\def\count@toks#1{\the\numexpr-1\count@@toks#1.\tc@endcnt}
\long\def\count@@toks#1#2\tc@endcnt{+1\tc@ifempty{#2}{\relax}{\count@@toks#2\tc@endcnt}}
\def\tc@ifempty#1{\tc@testxifx{\expandafter\relax\detokenize{#1}\relax}}
\long\def\tc@earg#1#2{\expandafter#1\expandafter{#2}}
\long\def\tctestifnum#1{\tctestifcon{\ifnum#1\relax}}
\long\def\tctestifcon#1{#1\expandafter\tc@exfirst\else\expandafter\tc@exsecond\fi}
\long\def\tc@testxifx{\tc@earg\tctestifx}
\long\def\tctestifx#1{\tctestifcon{\ifx#1}}
\long\def\tc@exfirst#1#2{#1}
\long\def\tc@exsecond#1#2{#2}
\makeatother

\DeclareMathOperator{\conv}{conv}
\DeclareMathOperator{\diag}{diag}

\DeclareMathOperator{\sign}{sign}

\newcommand{\ball}[2]{B_{#1}\ifblank{#2}{}{(#2)}
}

\newcommand{\field}[1]{\mathbb{#1}}
\newcommand{\nats}{\field{N}}

\newcommand{\reals}{\field{R}}

\newcommand{\Union}{\bigcup}

\newcommand{\downto}{\searrow}
\newcommand{\from}{\colon}

\newcommand{\oneto}[1]{\lbrack#1\rbrack}

\newcommand{\with}{\mid}
\newcommand{\pinv}{\dagger}
 \RequirePackage{mathtools}
\RequirePackage{etoolbox}

\DeclarePairedDelimiter{\paren}{\lparen}{\rparen}

\DeclarePairedDelimiter{\set}{\lbrace}{\rbrace}

\DeclarePairedDelimiter{\rlclosed}{\lbrack}{\rbrack}

\DeclarePairedDelimiterXPP{\abs}[1]
    {}
    {\lvert}{\rvert}
    {}
    {\ifblank{#1}{\bullet}{#1}}

\DeclarePairedDelimiterXPP{\norm}[1]
    {}
    {\lVert}{\rVert}
    {}
    {\ifblank{#1}{\bullet}{#1}}

\DeclarePairedDelimiterXPP{\seq}[2]
    {}
    {\lparen}{\rparen}
    {_{#2}}
    {#1}

\DeclarePairedDelimiterXPP{\family}[2]
    {}
    {\lparen}{\rparen}
    {_{#2}}
    {#1}

\DeclarePairedDelimiterXPP{\dprod}[2]
    {}
    {\langle}{\rangle}
    {}
    {\ifblank{#1}{\bullet}{#1}, \ifblank{#2}{\bullet}{#2}}

\DeclarePairedDelimiterXPP{\inner}[2]
    {}
    {\langle}{\rangle}
    {}
    {\ifblank{#1}{\bullet}{#1}, \ifblank{#2}{\bullet}{#2}}
 \RequirePackage{amsthm}
\makeatletter
\@ifclassloaded{beamer}
{}{\theoremstyle{plain}
    \newtheorem{globally}{DUMMY ENVIRONMENT FOR GLOBALLY NUMBERED ENVIRONMENTS}
    \newtheorem{corollary}[globally]{Corollary}
    \newtheorem{lemma}[globally]{Lemma}
    
    \newtheorem{theorem}[globally]{Theorem}

    \theoremstyle{definition}
    \newtheorem{definition}[globally]{Definition}
    \newtheorem{example}[globally]{Example}
    
}
\makeatother
 
\makeatletter
\@ifclassloaded{beamer}{\subimport{.}{beamer.tex}
}{}
\makeatother
\renewcommand{\oneto}[1]{\set{1,\ldots,#1}}
\newcommand{\fixed}[1]{\mathring{#1}}
\newcommand\logeq{\;\;\mathrel{\vcentcolon\Longleftrightarrow}\;\;}
\newcommand{\succmore}[1]{\mathcal{D}(#1)}
\NewDocumentCommand{\Cabs}{ O{2} }{C_{\operatorname{abs}}^{#1}}
\newcommand{\relu}{\operatorname{relu}}
\renewcommand{\bouligand}[1]{\partial_L#1}

\usepackage{ifthen}

\begin{document}

\maketitle

\begin{abstract}
    Algorithmic differentiation (AD) tools allow to obtain gradient information
    of a continuously differentiable objective function in a computationally
    cheap way using the so-called backward mode.
    It is common practice to use the same tools even in the absence of
    differentiability, although the resulting vectors may not be generalized
    gradients in the sense of Clarke.

    The paper at hand focuses on objectives in which the non-differentiability
    arises solely from the evaluation of the absolute value function.
    In that case, an algebraic condition based on the evaluation procedure of
    the objective is identified, that guarantees that Clarke gradients are
    correctly computed without requiring any modifications of the AD tool in
    question.
    The analysis allows to prove that any standard AD tool is adequate to
    drive a stochastic generalized gradient descent method for training a dense
    neural network with ReLU activations.
    The same is true for generalized batch gradients or the full generalized
    gradient, provided that the AD tool makes a deterministic and agnostic
    choice for the derivative information of the absolute value at \(0\).
\end{abstract}

\section{Introduction}
Non-smooth optimization problems appear in a wide range of applications.
Consequently, the analysis and implementation of suitable algorithms has been a
focal point in the mathematical optimization research.
In particular, in recent years the interest has risen due to the popularity of
non-smooth models in machine learning.
Many optimization algorithms require one or more generalized gradients at every
iterate.
The simplest example for such an algorithm, is the generalized (sub)gradient
descent method introduced in the 1960s, see~\cite[Chapter 2]{Sho85}.
More adaptions of smooth algorithms have been developed since, e.g.,
generalizations of the conjugate gradients method in~\cite{Wol75} and, more
recently, the non-smooth BFGS method~\cite{LO13}.
Other popular approaches are Shor's r-Algorithm~\cite{Sho70} as well as bundle
methods, see~\cite[Chapter 12]{BKM14}.
The book by Bagriov, Karmitsa and Mäkelä~\cite{BKM14} gives an excellent
overview over the basic algorithmic ideas to solve non-smooth optimization
problems.

While in the smooth case, a gradient oracle can be realized by means of
backward mode algorithmic differentiation (AD), see~\cite{GW08}, this is in
general not possible because chain-rule fails for generalized gradients.
To overcome this Khan and Barton in~\cite{KB12} proposed to use the
lexicographic differentiation from Nesterov~\cite{Nes05} in a forward mode
algorithmic differentiation approach.
They exploit chain-rule for (generalized) directional derivatives, however at
the price of giving up the cheap gradient principle of the backward mode.
For the abs-smooth functions, i.e., functions where non-smoothness arises
only due to absolute values, Griewank in~\cite{Gri13} has enhanced this to a
hybrid forward-backward approach which usually allows to obtain generalized
gradients in a much cheaper way.
ADOL-C~\cite{GJU96} is, to the best of the authors' knowledge, the only AD
framework able to perform this so-called polynomial escape.

Other AD tools pragmatically apply chain-rule for non-smooth functions,
using elements of the Clarke subdifferential at non-smooth evaluations.
This yields of course (generalized) gradients at smooth points, i.e., if no
non-smoothness was encountered in the evaluation procedure of the function.
However, in non-smooth points this might fail to compute generalized gradients
in the sense of Clarke.
Bolte and Pauwels in~\cite{BP21} have shown that standard forward and backward
mode AD computes generalized gradients almost everywhere, i.e., the set of
points where are AD fails to compute a Clarke gradient is of measure zero.
They proceed to analyze non-smooth (stochastic) gradient descent with so-called
conservative fields as surrogates for the Clarke differential, differing only
on the aforementioned nullset.
Elements of such set-valued functions are, unlike Clarke gradients, computable
by AD.

The aim of the paper is to show that, in the abs-smooth case, the linear
independence kink qualification (LIKQ) condition of~\cite{GW16} is sufficient to
guarantee that naive algorithmic differentiation computes a correct Clarke
gradient.
The LIKQ condition is a generic property that is satisfied by a broad class of
problems, provided that they are represented by reasonable evaluation
procedures.
Under the LIKQ condition it is possible to verify optimality conditions of
abs-smooth objective functions in polynomial time, e.g., in~\cite{GW19, FWG19,
Kre23}.
Abs-smooth problems can be reformulated as MPCC problems and in that case the
LIKQ condition translates to MPCC-LICQ, see~\cite{HKS21}.

This paper is structured as follows.
In Section~\ref{sec:absnormal}, the basics on abs-smooth functions in so-called
abs-normal form is discussed, in particular, this comprises the LIKQ condition
and the connection to AD.
In Section~\ref{sec:gradients}, we establish our desired results, showing that
under LIKQ, AD computes Clarke generalized gradients.
In some cases, even limiting gradients are obtained.
Section~\ref{sec:nns} then focuses on feed forward neural networks.
The LIKQ condition is verified for the computation of stochastic gradients of
training problems.
The paper is concluded by discussing batch gradients, proofing that practical AD
tools also compute generalized gradients there.

\section{Abs-normal Form and LIKQ}\label{sec:absnormal}
Let \(n, s \in \nats\), \(f \in C^2(\reals^{n + s + s})\) and
\(c \in C^2(\reals^{n + s + s}; \reals^s)\) such that for
any pair \((x, z) \in \reals^n \times \reals^s\) the partial derivatives
\(\partial_2 c(x, \abs{z}, z) \in \reals^{s \times s}\) and
\(\partial_3 c(x, \abs{z}, z) \in \reals^{s \times s}\) are strictly lower
triangular matrices.
By this assumption,
\(I_s - \partial_2 c(x, \abs{z}, z) - \partial_3 c(x, \abs{z}, z)\)
is invertible and the \emph{switching equation}
\begin{equation}
    \label{eq:switching}
    z = c(x, \abs{z}, z)
\end{equation}
has a unique solution \(z \in \reals^s\) for every \(x \in \reals^n\).
In particular, each component of \(z\) can be explicitly computed from \(x\) and
the previous components in \(z\).
The solution operator is denoted by \(z[\cdot] \colon \reals^n \to \reals^s\),
i.e., given \(x\) the pair \((x, z[x])\) is a solution to~\eqref{eq:switching}.
The two functions \(f\) and \(c\) can then be used to represent a single
function \(\varphi \from \reals^n \to \reals\) in the following sense.
If for all \(x \in \reals^n\)
\begin{equation}
    \varphi(x) = f(x, \abs{z[x]}, z[x])
\end{equation}
the tuple \((f, c)\) is then called \emph{evaluation procedure} of \(\varphi\)
and \(\varphi\) is called \emph{abs-smooth}.
The set of abs-smooth functions, i.e., the set of functions for which there is
an evaluation procedure is denoted by \(\Cabs(\reals^n)\).

Even though \(\varphi\) is in general non-smooth, due to the involvement of the
absolute value function, many tools for algorithmic differentiation (AD) are
able to perform backward mode differentiation for \(\varphi\) at \(x \in
\reals^n\).
In doing so, the absolute value function is treated as an elementary function
and a fixed element \(\xi \in \clarke{\abs{\cdot}}(z[x]) \subseteq
\rlclosed{-1, 1}^s\) of the Clarke generalized differential, see~\cite{Cla90},
is used for the required derivative information.
Thereby, effectively a gradient of a smooth function \(\varphi_\xi \in
C^2(\reals^n)\), with \(\varphi_\xi(x) = \varphi(x)\), is computed.
In the evaluation of \(\varphi_\xi\) the term \(\abs{z[x]}\) is replaced via
\(\abs{z[x]} = \Xi z[x]\) for \(\Xi \coloneqq \diag(\xi)\) which leads to
\begin{align}
    \label{eq:fixed-switching}
    z_\xi =&\ c(x, \Xi z_\xi, z_\xi) \\
    \text{and}\qquad\varphi_{\xi}(x) \coloneqq&\ f(x, \Xi z_\xi[x], z_\xi[x]),
\end{align}
where \(z_\xi[\cdot]\) denotes the solution operator to the modified switching
equation.
Table~\ref{tab:ad-tools} lists the choices for \(\xi_i \in \clarke
\abs{\cdot}(0)\) for different popular AD tools.
The result of this strategy in general will not lead to a Clark generalized
gradient of \(\varphi\) at \(x\), i.e., \(\grad \varphi_\xi(x) \notin \clarke
\varphi(x)\) because the Clarke differential in general does not obey a chain
rule, hence the choices for the components of \(\xi_i\) can not be made
independently of each other.

\begin{table}[ht]
    \centering
    \begin{tabular}{cccc}
        \toprule
        language & tool         & version                  & \(\xi_i \in \clarke \abs{\cdot}(0)\) \\
        \midrule
        \multirow{3}{*}{python} & jax~\cite{Bra+18}        & 0.4.20 & \(1.0\)              \\
                                & tensorflow~\cite{Aba+15} & 2.15.0 & \(0.0\)              \\
                                & pytorch~\cite{Pas+19}    & 2.2.0  & \(0.0\)              \\
        \midrule
        julia                   & ReverseDiff~\cite{Rev20} & 1.15.1 & \(1.0\)              \\
        \midrule
        C                       & ADOL-C~\cite{GJU96}     & 1.10.0 & \(0.0\)              \\
        \midrule
        C++                     & CodiPack~\cite{SAG19}    & 2.2.0  & \(0.0\)              \\
        \bottomrule
    \end{tabular}
    \caption{Elements of \(\clarke \abs{\cdot}(0)\) used in backward propagation for
        different AD-tools. ADOL-C is using the \emph{naive} mode, not
        polynomial escape.}\label{tab:ad-tools}
\end{table}

Fix \(\fixed{x}\) and \(\fixed{z} \coloneqq z[\fixed{x}]\) as its solution
to~\eqref{eq:switching} and define
\begin{align}
        a &\coloneqq \grad_1 f(\fixed{x}, \abs{\fixed{z}}, \fixed{z}) \in \reals^n,
      & Z &\coloneqq \jac_1 c(\fixed{x}, \abs{\fixed{z}}, \fixed{z}) \in \reals^{n \times s},
    \\
        b &\coloneqq \grad_2 f(\fixed{x}, \abs{\fixed{z}}, \fixed{z}) \in \reals^s,
      & L &\coloneqq \jac_2 c(\fixed{x}, \abs{\fixed{z}}, \fixed{z}) \in \reals^{s \times s},
    \\
        d &\coloneqq \grad_3 f(\fixed{x}, \abs{\fixed{z}}, \fixed{z}) \in \reals^s,
      & M &\coloneqq \jac_3 c(\fixed{x}, \abs{\fixed{z}}, \fixed{z}) \in \reals^{s \times s}.
\end{align}
Further, define the \emph{signature vector} \(\fixed{\sigma} \coloneqq
\sign(\fixed{z}) \in \set{-1, 0, 1}^s\), the \emph{signature matrix}
\(\fixed{\Sigma} \coloneqq \diag(\fixed{\sigma})\) and the set of active
indices \(\fixed{\alpha} = \set{i \with \fixed{\sigma}_i = 0}\).
The projection onto the active indices is defined by \(P_{\fixed{\alpha}}
\coloneqq \paren{e^T_i}_{i \in \fixed{\alpha}}\), where \(e_i \in \reals^s\) is
the \(i\)-th unit vector.

The signature vectors in \(\set{-1, 0, 1}^s\) are partially ordered by the
precedence relation
\begin{equation}
    \sigma \succeq \fixed{\sigma}
    \logeq \sigma_i \fixed{\sigma}_i \ge \fixed{\sigma}_i^2\, \text{ for all } i \in \oneto{s}.
\end{equation}
By this definition, \(\sigma \succeq \fixed{\sigma}\) holds if and only if
\(\sigma_i = \fixed{\sigma}_i\) whenever \(\fixed{\sigma}_i \ne 0\).
The set of all definite signatures, i.e., \(\sigma \in \set{-1,1} \), with \(\sigma \succeq
\fixed{\sigma}\) is denoted by
\begin{equation}
    \succmore{\fixed{\sigma}}
    \coloneqq \set{\sigma \in \set{-1, 1}^s \with \sigma \succeq \fixed{\sigma}}.
\end{equation}

\begin{definition}[Signature domains]
	For \(\sigma \in \set{-1, 0, 1}^s\) the set
	\begin{equation}
		S_\sigma \coloneqq \set{x \in \reals^n \with \sign(z[x])=\sigma}
	\end{equation}
	is called \emph{signature domain}. If \(S_\sigma\) is open and non-empty, it is called \emph{essential}.
\end{definition}

Signature domains provide a canonical way of writing functions in
\(\Cabs[2, s](\reals^n)\) as piecewise smooth functions in the sense of Scholtes
\cite{Sch12} because \(\varphi\) is a smooth function on every \(S_\sigma\).

\begin{lemma}\label{lem:consitent-solution}
	Let \(\sigma \in \set{-1, 0, 1}^s\).
	Then,
	\begin{equation}
		S_\sigma = \set{x \in \reals^n \with \sign(z_\sigma[x]) = \sigma}.
	\end{equation}
	If \(x \in S_\sigma\), then \(z_\sigma[x] = z[x]\) and \(\varphi_{\sigma}(x) =
	\varphi(x)\).
\end{lemma}

\begin{proof}
	Let \(\Sigma \coloneqq \diag(\sigma)\).
	Assume \(x \in \set{x \in \reals^n \with \sign(z_\sigma[x]) = \sigma}\),
	then \(\Sigma z_\sigma[x] = \abs{z_\sigma[x]}\).
	Thereby,
	\begin{equation}
		z_\sigma[x] = c(x, \Sigma z_\sigma[x], z_\sigma[x]) = c(x,\abs{z_\sigma[x]}, z_\sigma[x]).
	\end{equation}
	That is \(z_\sigma[x]\) satisfies~\eqref{eq:switching}, i.e., \(z_\sigma[x]
	= z[x]\), and hence, \(\sign(z[x]) = \sigma\).
	If otherwise \(x \in S_\sigma\), then
	\(\Sigma z[x] = \abs{z[x]}\) and
	\begin{equation}
		z[x] = c(x,\abs{z[x]}, z[x]) = c(x, \Sigma z[x], z[x]).
	\end{equation}
	This time \(z[x]\) satisfies~\eqref{eq:fixed-switching}, i.e., \(z[x] =
	z_\sigma[x]\), and as above it follows \(\sign{z_\sigma[x]} = \sigma\).

	Since in both cases \(z[x] = z_\sigma[x]\) it also follows
	\begin{align}
		\varphi(x)
		&= f(x, \abs{z[x]}, z[x])
		= f(x, \Sigma z_\sigma[x], z_\sigma[x])
		= \varphi_\sigma(x).
		\qedhere \end{align}
\end{proof}
If \(S_{\fixed{\sigma}}\) is essential and \(\fixed{\sigma} = \set{-1,1}^s\),
this means that \(\varphi\) was actually smooth at \(\fixed{x}\) and instead of
computing a generalized gradient of \(\varphi\) at \(\fixed{x}\) one can simply
compute a derivative of \(\varphi_{\fixed{\sigma}}\).
In the chain-rule, one then computes
\begin{align}
	\grad \varphi_{\fixed{\sigma}}(\fixed{x})
    &= a + \jac z_{\fixed{\sigma}}[\fixed{x}]^T (\fixed{\Sigma} b + d), \\
    \shortintertext{where}
	\jac z_{\fixed{\sigma}}[\fixed{x}]
    &= \paren{I_s - L\fixed{\Sigma} - M}^{-1} Z \in \reals^{s \times n}.
\end{align}
Even in the case that \(\fixed{\sigma}\) is not definite, the term \(\jac
z_{\fixed{\sigma}}[\fixed{x}]\) is interesting because it provides information
about the local structure of all \(S_\sigma\) with \(\sigma \succ \fixed{\sigma}\),
provided that the following condition holds.

\begin{definition}[linear independence kink qualification (LIKQ)]
    The evaluation procedure \((f, c)\) is said to satisfy the LIKQ condition
    at \(\fixed{x}\) if the matrix
	\begin{equation}
	    P_{\fixed{\alpha}} \paren{I_s - L\fixed{\Sigma} - M}^{-1} Z \in \reals^{\abs{\fixed{\alpha}} \times n}
	\end{equation}
	has full row rank \(\abs{\fixed{\alpha}}\).
\end{definition}

This condition implies that for all \(\sigma \in \succmore{\fixed{\sigma}}\)
with \(\fixed{x}\in \closure{S_\sigma}\), the standard LICQ constraint
qualification holds for \(\closure{S_\sigma}\).
This allows to check optimality conditions for \(\Cabs\)~functions in
polynomial time, which is used, e.g., in the active signature method for
piecewise linear problems~\cite{GW19}, the SALMIN method for general abs-smooth
problems~\cite{FWG19} as well as for constraint piecewise linear
problems~\cite{Kre23}, with an adapted version of LIKQ.
The LIKQ condition can be relaxed to the so-called MFKQ condition, which
corresponds to MFCQ condition on the closure on the adjacent signature domains,
see~\cite[Definition 2.12]{WG19}.

Knowing that \(c\) is twice continuously differentiable, the following lemma
shows that LIKQ is a stable property, in that moving from \(\fixed{x}\) to a
nearby \(x\) and thereby changing from \(\fixed{\sigma}\) to \(\sigma \succeq
\fixed{\sigma}\) does not impact the LIKQ condition.
\begin{lemma}\label{lem:full-rank-for-Sigma}
    The evaluation procedure \((f, c)\) has LIKQ at
    \(\fixed{x}\), if and only if, for all \(\sigma\in \set{-1,0,1}^s\) with
    \(\sigma \succeq \fixed{\sigma}\),
    \begin{equation}\label{eq:full-rank-for-Sigma}
        P_{\fixed{\alpha}}\paren{I_s - L \Sigma - M}^{-1} Z,
    \end{equation}
    where \(\Sigma \coloneqq \diag(\sigma)\), has full row rank
    \(\abs{\fixed{\alpha}}\).
\end{lemma}
\begin{proof}
    Insert \(\pm L \fixed{\Sigma}\) to the middle matrix
    of~\eqref{eq:full-rank-for-Sigma} and rewrite the expression in terms of
    \(\tilde{L} \coloneqq \paren{I_s - L \fixed{\Sigma} - M}^{-1} L\)
    to obtain
    \begin{equation}
        \begin{aligned}
            I_s - L \Sigma - M
            &= I_s - L \fixed{\Sigma} - M - L(\Sigma - \fixed{\Sigma})\\
            &= (I_s - L \fixed{\Sigma} - M)(I_s - \paren{I_s - L \fixed{\Sigma} - M}^{-1} L(\Sigma - \fixed{\Sigma}))\\
            &= (I_s - L \fixed{\Sigma} - M)(I_s - \tilde{L}(\Sigma - \fixed{\Sigma})),
        \end{aligned}
    \end{equation}
    which was already observed in~\cite{GW20b}.
    A multiplication of the inverse with \(P_{\fixed{\alpha}}\) from the left
    thus yields
    \begin{equation}\label{eq:step1:lem:full-rank-for-Sigma}
        P_{\fixed{\alpha}} \paren{I_s - L \Sigma - M}^{-1}
        = P_{\fixed{\alpha}} \paren{I_s - \tilde{L} (\Sigma - \fixed{\Sigma})}^{-1} \paren{I_s - L \fixed{\Sigma} - M}^{-1}.
    \end{equation}
    Let \(Q_{\fixed{\alpha}} \coloneqq \paren{e_i^T}_{i \in \oneto{s}
    \setminus \fixed{\alpha}}\) be the complementary projection to the inactive
    indices, then
    \begin{equation}\label{eq:active-inactive-properties}
        P_{\fixed{\alpha}}P_{\fixed{\alpha}}^T = I_{\abs{\fixed{\alpha}}},
        \; P_{\fixed{\alpha}}Q_{\fixed{\alpha}}^T = 0,
        \; Q_{\fixed{\alpha}}P_{\fixed{\alpha}}^T = 0,
        \; Q_{\fixed{\alpha}}Q_{\fixed{\alpha}}^T = I_{s - \abs{\fixed{\alpha}}},
        \;\text{and}\;
        \begin{bmatrix}
            P_{\fixed{\alpha}}^T & Q_{\fixed{\alpha}}^T
        \end{bmatrix}
        =
        \begin{bmatrix}
            P_{\fixed{\alpha}} \\ Q_{\fixed{\alpha}}
        \end{bmatrix}^{-1}.
    \end{equation}
    The last equality is used to insert an identity in the right-hand side
    of~\eqref{eq:step1:lem:full-rank-for-Sigma} and obtain
    \begin{equation}\label{eq:step2:lem:full-rank-for-Sigma}
        \begin{aligned}
            P_{\fixed{\alpha}} \paren{I_s - L \Sigma - M}^{-1}
            &= P_{\fixed{\alpha}}
                \paren{I_s - \tilde{L} (\Sigma - \fixed{\Sigma})}^{-1}
                \begin{bmatrix} P_{\fixed{\alpha}}^T & Q_{\fixed{\alpha}}^T \end{bmatrix}
                \begin{bmatrix} P_{\fixed{\alpha}} \\ Q_{\fixed{\alpha}} \end{bmatrix}
                \paren{I_s - L \fixed{\Sigma} - M}^{-1}\\
            &= P_{\fixed{\alpha}}
                \paren*{\begin{bmatrix} P_{\fixed{\alpha}} \\ Q_{\fixed{\alpha}} \end{bmatrix}
                        - \begin{bmatrix} P_{\fixed{\alpha}} \\ Q_{\fixed{\alpha}} \end{bmatrix}
                          \tilde{L} (\Sigma - \fixed{\Sigma})}^{-1}
                          \begin{bmatrix} P_{\fixed{\alpha}}\paren{I_s - L \fixed{\Sigma} - M}^{-1} \\
                                          Q_{\fixed{\alpha}}\paren{I_s - L \fixed{\Sigma} - M}^{-1}
                          \end{bmatrix}.
        \end{aligned}
    \end{equation}
    From \(\sigma_{i} = \fixed{\sigma}_{i}\) for \(i \in \oneto{s} \setminus
    \fixed{\alpha}\) it follows \((\Sigma - \fixed{\Sigma})
    Q_{\fixed{\alpha}}^T = 0\), and hence,
    \begin{equation}\label{eq:step3:lem:full-rank-for-Sigma}
        \begin{aligned}
            P_{\fixed{\alpha}}
            \paren*{\begin{bmatrix} P_{\fixed{\alpha}} \\ Q_{\fixed{\alpha}} \end{bmatrix}
                    - \begin{bmatrix} P_{\fixed{\alpha}} \\ Q_{\fixed{\alpha}} \end{bmatrix}
                      \tilde{L} (\Sigma - \fixed{\Sigma})}^{-1}
&= \begin{bmatrix} I_{\abs{\fixed{\alpha}}} & 0 \end{bmatrix}
                \paren*{I_s - \begin{bmatrix} P_{\fixed{\alpha}} \\ Q_{\fixed{\alpha}} \end{bmatrix}
                              \tilde{L} (\Sigma - \fixed{\Sigma})
                              \begin{bmatrix} P_{\fixed{\alpha}}^T & Q_{\fixed{\alpha}}^T \end{bmatrix}}^{-1}\\
            &= \begin{bmatrix} I_{\abs{\fixed{\alpha}}} & 0 \end{bmatrix}
                \begin{bmatrix}
                    I_{\abs{\fixed{\alpha}}} -P_{\fixed{\alpha}} \tilde{L} (\Sigma - \fixed{\Sigma}) P_{\fixed{\alpha}}^T & 0 \\
                    Q_{\fixed{\alpha}} \tilde{L} (\Sigma - \fixed{\Sigma}) P_{\fixed{\alpha}}^T & I_{s - \abs{\fixed{\alpha}}} \\
                \end{bmatrix}^{-1}\\
            &\eqqcolon
            \begin{bmatrix}
                X_1 & X_2
            \end{bmatrix}.
        \end{aligned}
    \end{equation}
    The definition of the two matrices \(X_1 \in \reals^{\abs{\fixed{\alpha}}
    \times \abs{\fixed{\alpha}}}\) and \(X_2 \in \reals^{\abs{\fixed{\alpha}}
    \times (s- \abs{\fixed{\alpha}})}\) implies
    \begin{equation}
        \begin{bmatrix}
            X_1 & X_2
        \end{bmatrix}
        \begin{bmatrix}
            I_{\abs{\fixed{\alpha}}} -P_{\fixed{\alpha}} \tilde{L} (\Sigma - \fixed{\Sigma}) P_{\fixed{\alpha}}^T & 0 \\
            Q_{\fixed{\alpha}} \tilde{L} (\Sigma - \fixed{\Sigma}) P_{\fixed{\alpha}}^T & I_{s - \abs{\fixed{\alpha}}} \\
        \end{bmatrix}
        =
        \begin{bmatrix}
            I & 0
        \end{bmatrix},
    \end{equation}
    which leads to \(X_2 = 0\) and \(X_1 = \paren{I_{\abs{\fixed{\alpha}}} - P_{\fixed{\alpha}}
    \tilde{L} (\Sigma - \fixed{\Sigma}) P_{\fixed{\alpha}}^T}^{-1}\).
    Combining this
    with~\eqref{eq:step1:lem:full-rank-for-Sigma},~\eqref{eq:step2:lem:full-rank-for-Sigma}
    and~\eqref{eq:step3:lem:full-rank-for-Sigma} and finally multiplying with
    \(Z\) from the right yields
    \begin{align}
        P_{\fixed{\alpha}} \paren{I_s - L \Sigma - M}^{-1} Z
        &=
        \begin{bmatrix} \paren{I_{\abs{\fixed{\alpha}}} - P_{\fixed{\alpha}} \tilde{L} (\Sigma - \fixed{\Sigma}) P_{\fixed{\alpha}}^T}^{-1} & 0 \end{bmatrix}
        \begin{bmatrix} P_{\fixed{\alpha}}\paren{I_s - L \fixed{\Sigma} - M}^{-1} \\ Q_{\fixed{\alpha}}\paren{I_s - L \fixed{\Sigma} - M}^{-1} \end{bmatrix} Z\\
        &= \paren{I_{\abs{\fixed{\alpha}}} - P_{\fixed{\alpha}} \tilde{L} (\Sigma - \fixed{\Sigma}) P_{\fixed{\alpha}}^T}^{-1}P_{\fixed{\alpha}}\paren{I_s - L \fixed{\Sigma} - M}^{-1} Z.
    \end{align}
    Since transformation by a full rank matrix does not change the rank, this
    shows the desired identity by the definition of LIKQ.
\end{proof}

\section{Computing Generalized Gradients under LIKQ}\label{sec:gradients}
The aim of this section is to characterize generalized gradients of the
abs-smooth \(\varphi\) at \(\fixed{x}\) in terms of the gradients \(\grad
\varphi_\sigma(\fixed{x})\), \(\sigma \in \succmore{\fixed{\sigma}}\).
\begin{definition}[limiting gradients and Clark generalized gradients]
    Let \(\varphi \colon \reals^n \to \reals\) be locally Lipschitz continuous.
    The set of \emph{limiting gradients} of \(\varphi\) at \(\fixed{x} \in \reals^n\) is defined
    by
    \begin{equation}
        \bouligand{\varphi}(\fixed{x})
        \coloneqq
        \set{g \in \reals^n \with g = \lim \grad{\varphi}(x_k), x_k \in D_{\varphi}, x_k \to \fixed{x}},
    \end{equation}
    where \(D_{\varphi}\) is the set of differentiable points of \(\varphi\).
    Furthermore,
    \begin{equation}
        \clarke{\varphi}(\fixed{x}) \coloneqq \conv\bouligand{\varphi}(\fixed{x})
    \end{equation}
    denotes the set of all \emph{Clarke generalized gradients} of \(\varphi\) at
    \(\fixed{x}\).
    By~\cite[Theorem~2.5.1]{Cla90} both sets \(\bouligand{\varphi}(\fixed{x})\)
    and \(\clarke{\varphi}(\fixed{x})\) are non-empty.
\end{definition}
In~\cite[Proposition 4.3.1]{Sch12}, Scholtes essentially proofs
\begin{equation}\label{eq:bouligand_scholtes}
	\bouligand{\varphi}(\fixed{x})
	= \Union_{\mathclap{\sigma \in \mathcal{E}_e(\fixed{x})}}\; \set{\grad \varphi_\sigma(\fixed{x})},
\end{equation}
where \(\mathcal{E}_e(\fixed{x})\) is the set of essentially active signatures,
i.e., the set of all \(\sigma \in \set{-1,0,1}^s\) such that \(\fixed{x}\in
\closure{\operatorname{int}(S_\sigma)}\) and \(S_\sigma\) is essential.
In~\cite{Gri13}, the concept of conical gradients was introduced.
These are the gradients of all \(\varphi_\sigma(\fixed{x})\) for which the
linearized cones of the signature domains \(S_\sigma\) have non-empty interior.
The set of all \emph{conically active} signatures is denoted by
\(\mathcal{E}_c(\fixed{x})\).
It was already shown in~\cite{Gri13} that \(\mathcal{E}_c \subseteq \mathcal{E}_e\).
This was used to proof that conical gradients, which can be computed by
polynomial escape, see~\cite[Proposition~6]{Gri13}, are always limiting and
thus Clarke gradients.
Under the aforementioned MFKQ, all linearized cones remain essential if the
original \(S_\sigma\) was, hence, \(\mathcal{E}_c = \mathcal{E}_e\), see
\cite{WG19}.

Nevertheless, characterizing \(\mathcal{E}_e(\fixed{x})\) algebraically for general
\(\Cabs\) functions remains a non-trivial task. However, the first major result
of this section, Corollary~\ref{cor:bouligand}, shows that under the assumption
of LIKQ, \(\mathcal{E}_e(\fixed{x})\) can be replaced by
\(\succmore{\fixed{\sigma}}\). The significance of this is that in backward mode
algorithmic differentiation any choice of \(\partial \abs{\cdot}(0)=\pm 1\)
yields a limiting gradient because, under LIKQ, \(\grad \varphi_\sigma(\fixed{x})
\in \bouligand{\varphi}(\fixed{x})\) holds for all \(\sigma \in
\succmore{\fixed{\sigma}}\). This is extended to choices between \(-1\) and \(1\)
yielding Clarke generalized gradients in Theorem~\ref{thm:ad-grads-are-clarke}.
The following lemma is a key ingredient in proving these identities.

\begin{lemma}\label{lem:all-essential}
    Assume that the evaluation procedure \((f, c)\) of \(\varphi \in \Cabs[2,
    s](\reals^n)\) has LIKQ at \(\fixed{x}\).
    Then, there is \(r > 0\) such that for \(\tilde{x} \in
    \ball{r}{\fixed{x}}\) and all \(\sigma \in \succmore{\sign(z[\tilde{x}])}\)
    the signature domain \(S_\sigma\) is essential,
    \(\succmore{\sign(z[\tilde{x}])} \subseteq \succmore{\fixed{\sigma}}\), and
    \(\tilde{x} \in \closure{S_\sigma}\).
\end{lemma}

\begin{proof}
    Let \(\tilde{x} \in \ball{r}{\fixed{x}}\) and define \(\tilde{\sigma}
    \coloneqq \sign(z[\tilde{x}])\), \(\tilde{\Sigma} =
    \diag(\tilde{\sigma})\) as well as \(\tilde{\alpha} \coloneqq \set{i \with
    \tilde{\sigma} = 0}\).
    For \(r > 0\) small enough the continuity of \(z[\cdot]\) ensures that if
    \(\fixed{z}_i \neq 0\), it holds \(\sign(z[\tilde{x}]_i) =
    \sign(z[\fixed{x}]_i)\).
    Therefore, \(\tilde{\sigma} \succeq \fixed{\sigma}\) and since \(\succeq\) is a
    partial ordering, it holds \(\succmore{\tilde{\sigma}} \subseteq
    \succmore{\fixed{\sigma}}\).
	This also implies that \(\tilde{\alpha} \subseteq \fixed{\alpha}\).

	Now let \(\sigma \in \succmore{\tilde{\sigma}} \subseteq \succmore{\fixed{\sigma}}\).
    With LIKQ, Lemma~\ref{lem:full-rank-for-Sigma} ensures that the matrix
    \(P_{\fixed{\alpha}}\paren{I - L\Sigma - M}^{-1}Z\) has full row rank.
    Using \(\tilde{\alpha} \subseteq \fixed{\alpha}\), the smaller matrix
    \(P_{\tilde{\alpha}}\paren{I - L\Sigma - M}^{-1}Z\) also has full row rank
    \(\abs{\tilde{\alpha}}\).
 	For \(\varepsilon > 0\) define \(d_\sigma \coloneqq \paren{P_{\tilde{\alpha}}\paren{I - L\Sigma -
    M}^{-1}Z}^\pinv P_{\tilde{\alpha}}\sigma\) and \(x_{\sigma, \varepsilon}
    \coloneqq \tilde{x} + \varepsilon d_\sigma\).
    By Taylor's theorem for \(z_\sigma[\cdot]\) it follows
    \begin{equation}
    	\label{eq:taylor_for_xeps}
        \begin{aligned}
            z_\sigma[x_{\sigma, \varepsilon}]
            &= z_\sigma[\tilde{x}] + \varepsilon \jac z_\sigma[\tilde{x}] d_\sigma + \mathcal{O}(\varepsilon^2)\\
            &= z_\sigma[\tilde{x}] + \varepsilon \jac z_\sigma[\fixed{x}] d_\sigma + \varepsilon (\jac z_\sigma[\tilde{x}] - \jac z_\sigma[\fixed{x}]) d_\sigma + \mathcal{O}(\varepsilon^2).
        \end{aligned}
    \end{equation}
	By the implicit function theorem, it holds
	\begin{equation}
		\jac z_\sigma[\fixed{x}]
		= \paren{I - L\Sigma - M}^{-1}Z.
	\end{equation}
    For \(i \in \tilde{\alpha}\) the definitions of \(P_{\tilde{\alpha}}\) and
    \(d_\sigma\) yield
	\begin{equation}
        \label{eq:sign-ztilde-active}
        (\jac z_\sigma[\fixed{x}] d_\sigma)_i
        = (P_{\tilde{\alpha}}^T P_{\tilde{\alpha}} \jac z_\sigma[\fixed{x}] d_\sigma)_i
        = (P_{\tilde{\alpha}}^T P_{\tilde{\alpha}} \sigma)_i
        = \sigma_i.
	\end{equation}
    Moreover, Lemma~\ref{lem:consitent-solution} and
    \(z_{\tilde{\sigma}}[\tilde{x}] = z_{\sigma}[\tilde{x}]\), since
    \(\tilde{\Sigma}z_{\tilde{\sigma}}[\tilde{x}] = \Sigma
    z_{\tilde{\sigma}}[\tilde{x}]\), ensure \(0 = z[\tilde{x}]_i =
    z_{\tilde{\sigma}}[\tilde{x}]_i = z_\sigma[\tilde{x}]_i\).
    This together
    \newcommand{\LipDz}{\operatorname{Lip}(\jac z_\sigma)}
    with~\eqref{eq:taylor_for_xeps},~\eqref{eq:sign-ztilde-active}, and the
    Lipschitz-continuity of \(\jac z_\sigma[\cdot]\), with Lipschitz constant
    \(\LipDz\), shows that
    \begin{equation}
        \label{eq:correct-prescribed-sign}
        \begin{aligned}
            \sigma_i z_\sigma[x_{\sigma, \varepsilon}]_i
&= \sigma_i (z_\sigma[x_{\sigma, \varepsilon}]_i - z_\sigma[\tilde{x}]_i)\\
            &= \varepsilon \sigma_i (\jac z_\sigma[\fixed{x}] d_\sigma)_i + \varepsilon \sigma_i ((\jac z_\sigma[\tilde{x}] - \jac z_\sigma[\fixed{x}]) d_\sigma)_i + \mathcal{O}(\varepsilon^2)\\
&\ge \varepsilon \sigma_i^2 - \varepsilon \LipDz \norm{\tilde{x} - \fixed{x}}\norm{d_\sigma}  + \mathcal{O}(\varepsilon^2)\\
            &\geq \varepsilon (1 -  \LipDz r\norm{d_\sigma}) + \mathcal{O}(\varepsilon^2).
        \end{aligned}
    \end{equation}
    For \(r >0\) sufficiently small such that \(\LipDz r \norm{d_\sigma} \le
    1/2\), this yields
	\begin{equation}
	   \sigma_i z_\sigma[x_{\sigma, \varepsilon}]_i
       \ge \varepsilon (1 -  \LipDz r\norm{d_\sigma}) + \mathcal{O}(\varepsilon^2)  \geq \tfrac\varepsilon2 + \mathcal{O}(\varepsilon^2). \\
	\end{equation}
    Therefore, for all \(\varepsilon\) sufficiently small it holds \(\sigma_i
    z_\sigma[x_{\sigma, \varepsilon}]_i > 0\), hence, \(\sign(z_\sigma[x_{\sigma,
    \varepsilon}]_i)=\sigma_i\).
    For \(i \in \oneto{s} \setminus \tilde{\alpha}\), the latter also holds
    because \(z_\sigma[\cdot]\) depends continuously on its argument.
    Altogether again with Lemma~\ref{lem:consitent-solution}, it holds \(
    \sigma = \sign(z_\sigma[x_{\sigma, \varepsilon}])=  \sign(z[x_{\sigma,
    \varepsilon}])\), and thus, \(x_{\sigma, \varepsilon} \in S_\sigma\) all
    \(\varepsilon\) sufficiently small.
    In particular, \(S_\sigma\) is non-empty. Taking the limit \(\varepsilon \to
    0\) shows \(\tilde{x} \in \closure{S_\sigma}\).

    Now let \(x \in S_\sigma\) be arbitrary.
    Since \(z_\sigma[\cdot]\) is continuous, there exists a \(0 < \delta \in
    \reals\) such that \(\ball{\delta}{x} \subseteq S_\sigma\), and hence,
    \(S_\sigma\) is open.
\end{proof}
For \(\tilde{x}=\fixed{x}\) the above lemma proofs that all adjacent signature
domains with definite signatures are essential. However, this does not directly
ensure that there are no essential signature domains with vanishing signatures.
To obtain this result it is necessary to exploit the smoothness of \(c\) and
provide the above lemma also for all \(\tilde{x}\) sufficiently close to
\(\fixed{x}\). This can now be used in the following corollary, which is inspired
by the proof of Scholtes in~\cite{Sch12} with slight adaptions for the setting
at hand.

\begin{corollary}[Scholtes]
    \label{cor:bouligand}
    Assume that the evaluation procedure \((f, c)\) of \(\varphi \in \Cabs[2,
    s](\reals^n)\) has LIKQ at \(\fixed{x}\).
    Then,
    \begin{equation}
        \bouligand{\varphi}(\fixed{x})
        = \Union_{\mathclap{\sigma \in \succmore{\fixed{\sigma}}}}\; \set{\grad \varphi_\sigma(\fixed{x})}.
    \end{equation}
\end{corollary}

\begin{proof}
    Let \(\sigma \in \succmore{\fixed{\sigma}}\), then by
    Lemma~\ref{lem:all-essential}, \(S_\sigma\) is essential and \(\fixed{x}\in
    \closure{S_\sigma}\), hence, there is a sequence \(\seq{x_k}{k \in
    \nats}\) on \(S_\sigma\) such that \(x_k \to \fixed{x}\).
    Lemma~\ref{lem:consitent-solution} ensures that \(\varphi_\sigma =
    \varphi\) on \(S_\sigma\), hence \(\varphi\) is smooth and \(\grad \varphi_\sigma =
    \grad \varphi\) on \(S_\sigma\).
    This implies
    \begin{equation}
        \grad \varphi_\sigma(\fixed{x})
= \lim_{k \to \infty} \grad \varphi_\sigma(x_k)
        = \lim_{k \to \infty} \grad \varphi(x_k)
        \in \bouligand{\varphi}(\fixed{x}).
    \end{equation}

	Let now \(g \in \bouligand{\varphi}(\fixed{x})\).
    Then, there exists a sequence \(\seq{x_k}{k \in \nats}\) on \(\reals^d\)
    with \(x_k \to \fixed{x}\) such that \(\varphi\) is smooth at every \(x_k\)
    and \(\grad \varphi(x_k) \to g\).
    Let \(r > 0\) such that Lemma~\ref{lem:all-essential} holds for \(\fixed{x}\).
    Now for \(k\) large enough such that all \(x_k \in B_r(\fixed{x})\), it
    holds that \(x_k \in \closure{S_{\sigma_k}}\) for all \(\sigma_k \in
    \succmore{\sign(z[x_k])} \subseteq \succmore{\fixed{\sigma}}\).
    In particular, because \(\succmore{\sign(z[x_k])} \neq \emptyset\), there is
    always such a \(\sigma_k \in \succmore{\fixed{\sigma}}\) for \(x_k \in
    B_r(\fixed{x})\).

    However, there are only a finite number of elements in
    \(\succmore{\fixed{\sigma}}\).
    Hence, there is \(\sigma \in \succmore{\fixed{\sigma}}\) such that
    there is a subsequence \(\seq{x_{k_l}}{l \in \nats}\) with \(x_{k_l} \in
    \closure{S_{\sigma}}\).

    Lemma~\ref{lem:consitent-solution} ensures that for all \(x\in
    S_{\sigma}\), \(\varphi(x) = \varphi_{\sigma}(x)\).
    Since by Lemma~\ref{lem:all-essential}, \(S_{\sigma}\) is essential,
    \(\grad \varphi(x) = \grad \varphi_{\sigma}(x)\) holds for all \(x\in
    S_{\sigma}\).
    For any \(x_{k_l}\in \closure{S_{\sigma}}\), this yields \(\grad
    \varphi_{\sigma}(x_{k_l}) \in \bouligand{\varphi}(x_{k_l}) = \set{\grad
    \varphi(x_{k_l})}\) because \(\varphi\) is smooth at \(x_{k_l}\).
    This implies \(\grad \varphi_{\sigma}(x_{k_l}) = \grad \varphi(x_{k_l})\), and
    taking the limit yields \(\grad \varphi_{\sigma}(\fixed{x}) = g\).
\end{proof}

Without the LIKQ condition one does in general not obtain a limiting gradient
as the following example shows.
\begin{example}
    Let \(n = 2\), \(s = 3\) and \(\mu \in \reals\), and define
    \(f^\mu(x, y, z) \coloneqq y_1 + y_2 + \mu y_3\) and
    \begin{equation}
        c(x, y, z) \coloneqq
        \begin{pmatrix}
            x_1 \\
            x_2 - x_1 \\
            1 - \cos(x_1) - \sin(x_2)
        \end{pmatrix}.
    \end{equation}
    The overall function \(\varphi^\mu \in \Cabs[2, s](\reals^2)\) that is
    represented by \((f^\mu, c)\) then reads
    \begin{equation}
        \varphi^\mu(x)
        = \abs{x_1} + \abs{x_2 - x_1} + \mu \abs{1 - \cos(x_1) - \sin(x_2)}.
    \end{equation}
    \begin{figure}[htpb]
        \begin{subfigure}[b]{\textwidth/3}
            \centering
            \includegraphics{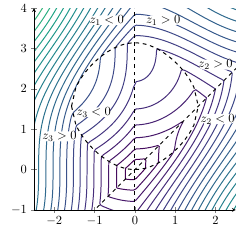}
            \caption{\(\mu = 1\)}
            \label{fig:contours-mu-one}
        \end{subfigure}\begin{subfigure}[b]{\textwidth/3}
            \centering
            \includegraphics{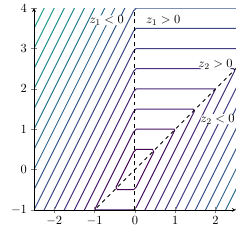}
            \caption{\(\mu = 0\)}
            \label{fig:contours-mu-zero}
        \end{subfigure}\begin{subfigure}[b]{\textwidth/3}
            \centering
            \includegraphics{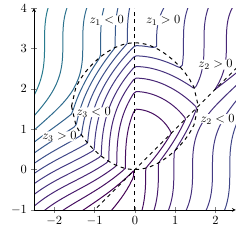}
            \caption{\(\mu = -1\)}
            \label{fig:contours-mu-minus-one}
        \end{subfigure}
        \caption{Level sets of \(\varphi^\mu\) for different values of \(\mu\).
            Points of non-differentiability lie along the dashed lines.
        }
        \label{fig:contours-example-one}
    \end{figure}

    Fix \(\fixed{x} \coloneqq (0, 0)\) it holds \(z[\fixed{x}] = (0, 0, 0)\),
    \(\fixed{\sigma} = (0, 0, 0)\) and \(\succmore{\fixed{\sigma}} = \set{-1,
    1}^3\), and \(\fixed{\alpha} = \set{1, 2, 3}\), i.e., \(P_{\fixed{\alpha}}
    = I_3\).
    Moreover, it holds that \(L = M = 0\) and
    \begin{equation}
        Z =
        \begin{pmatrix}
            1 & 0 \\
            -1 & 1 \\
            0 & -1
        \end{pmatrix},
    \end{equation}
    such that \(P_{\fixed{\alpha}} \paren{I_3 - L\diag(\fixed{\sigma}) -
    M}^{-1}Z = Z\) is of rank \(2 < \abs{\fixed{\alpha}} = 3\).
    Hence, \(\varphi^\mu = (f^\mu, c)\) does not have LIKQ at \(\fixed{x}\).
    For \(\sigma \in \succmore{\fixed{\sigma}}\) the smooth function
    \(\varphi^\mu_\sigma\) on reads
    \begin{equation}
        \varphi^\mu_\sigma(x) = \sigma_1 x_1 + \sigma_2(x_2 - x_1) + \mu \sigma_3(1 - \cos(x_1) - \sin(x_2)),\\
    \end{equation}
    and its gradient at \(\fixed{x}\) is given by
    \begin{equation}
        \grad \varphi^\mu_\sigma(\fixed{x}) =
        \begin{pmatrix}
            \sigma_1 - \sigma_2 \\
            \sigma_2 - \mu \sigma_3
        \end{pmatrix}.
    \end{equation}
    Figure~\ref{fig:grads-mu} illustrates the limiting gradients
    \(\bouligand{\varphi^\mu}(\fixed{x})\), the Clarke gradients
    \(\clarke{\varphi^\mu}(\fixed{x})\), and the gradients \(\grad
    \varphi^\mu_\sigma(\fixed{x})\) obtained by backward mode differentiation for
    all possible choices of \(\sigma \in \set{-1, 1}^3\).
    For \(\mu\geq0\), chain-rule from convex analysis holds.
    Therefore, a convex subgradient, i.e., a Clark generalized gradient, is
    obtained for all signatures.
    In the non-convex setting, however, the signatures \((1,1,1)\) and
    \((-1,-1,-1)\) do not yield Clarke generalized gradients. For \(\sigma =
    (-1,-1,-1)\) this is due to \(S_{(-1,-1,-1)}\) begin empty, see
    Figure~\ref{fig:contours-example-one}.
    Although \(S_{(1,1,1)}\) is non-empty, \(\grad
    \varphi^\mu_{(1,1,1)}(\fixed{x})\) is no limiting gradient because
    \(\fixed{x} \notin \closure{S_{(1,1,1)}}\).
    It is easy to see that the other signatures satisfy \(\fixed{x}\in
    \closure{S_\sigma}\) with MFCQ.
    Therefore, MFKQ as defined in~\cite[Definition~2.12]{WG19} holds in this
    example.
\begin{figure}[htpb]
        \tikzset{bgrad/.style = {
                inner sep=0pt,
                outer sep=0pt,
                circle,
                fill,
                minimum size=4.2pt,
                opacity=1,
                black,
            },
            adgrad/.style = {
                bgrad,
                minimum size=2.8pt,
                red,
            },
            gradlabel/.style = {
                black,
                inner sep=0pt,
                label distance=4pt,
                opacity=1,
                outer sep=0pt,
                right,
            },
            clarkegrads/.style = {
                fill,
                gray,
                opacity=.5,
            }
        }

        \newcommand{\drawgradients}[1]{
            \draw[->, dotted] (-2.4, 0) -- (2.6, 0);
            \draw[->, dotted] (0, -2.4) -- (0, 2.6);
            \foreach \muval in {#1}{
                \foreach \sigmaone in {-1, 1} {
                    \foreach \sigmatwo in {-1, 1} {
                        \foreach \sigmathree in {-1, 1} {
                            \coordinate (sigma1=\sigmaone_sigma2=\sigmatwo_sigma3=\sigmathree_mu=#1)
                                at (\sigmaone - \sigmatwo, \sigmatwo - \muval*\sigmathree) {};
                        }
                    }
                }
            }
            \coordinate (sigma1=1_sigma2=1_sigma3=1_mu=#1) at (0, 1 - #1);
            \coordinate (sigma1=-1_sigma2=-1_sigma3=-1_mu=#1) at (0, -1 + #1);
        }

        \begin{subfigure}[b]{\textwidth/3}
            \centering
            \tikzsetnextfilename{example-one-grads-mu-one}
            \begin{tikzpicture}[scale=.74,]
                \drawgradients{1}
                \draw[clarkegrads]
                    (sigma1=1_sigma2=1_sigma3=-1_mu=1) node[bgrad, label={[gradlabel]90:\tiny{$(1,\!1,\!-\!1)$}}] {} --
                    (sigma1=-1_sigma2=1_sigma3=-1_mu=1) node[bgrad, label={[gradlabel]90:\tiny{$(-\!1,\!1,\!-\!1)$}}] {} --
                    (sigma1=-1_sigma2=1_sigma3=1_mu=1) node[bgrad, label={[gradlabel]90:\tiny{$(-\!1,\!1,\!1)$}}] {} --
                    (sigma1=-1_sigma2=-1_sigma3=1_mu=1) node[bgrad, label={[gradlabel]-90:\tiny{$(-\!1,\!-\!1,\!1)$}}] {} --
                    (sigma1=1_sigma2=-1_sigma3=1_mu=1) node[bgrad, label={[gradlabel]-90:\tiny{$(1,\!-\!1,\!1)$}}] {} --
                    (sigma1=1_sigma2=-1_sigma3=-1_mu=1) node[bgrad, label={[gradlabel]90:\tiny{$(1,\!-\!1,\!-\!1)$}}] {} --
                    cycle;
                \node[adgrad, label={[gradlabel]90:\tiny{$(1,\!1,\!1)$}}] at (sigma1=1_sigma2=1_sigma3=1_mu=1) {};
                \node[adgrad, label={[gradlabel]-90:\tiny{$(\!-\!1,\!-\!1,\!-\!1)$}}] at (sigma1=-1_sigma2=-1_sigma3=-1_mu=1) {};
            \end{tikzpicture}
            \caption{\(\mu = 1\)}
            \label{fig:grads-mu-one}
        \end{subfigure}\begin{subfigure}[b]{\textwidth/3}
            \centering
            \tikzsetnextfilename{example-one-grads-mu-zero}
            \begin{tikzpicture}[
                scale=.74,
                ]
                \drawgradients{0}
                \draw[clarkegrads]
                (sigma1=1_sigma2=1_sigma3=-1_mu=0) node[bgrad, label={[gradlabel, label distance=11pt]90:\tiny{$(1,\!1,\!-\!1)$}}] {} --
                (sigma1=-1_sigma2=1_sigma3=1_mu=0) node[bgrad, label={[gradlabel]90:\tiny{$(-\!1,\!1,\!1)$}}] {} --
                (sigma1=-1_sigma2=-1_sigma3=1_mu=0) node[bgrad, label={[gradlabel]-90:\tiny{$(-\!1,\!-\!1,\!1)$}}] {} --
                (sigma1=1_sigma2=-1_sigma3=-1_mu=0) node[bgrad, label={[gradlabel, label distance=11pt]90:\tiny{$(1,\!-\!1,\!-\!1)$}}] {} --
                cycle;
                \node[bgrad, label={[gradlabel]90:\tiny{$(1,\!-\!1,\!1)$}}] at (sigma1=1_sigma2=-1_sigma3=1_mu=0) {};
                \node[bgrad, label={[gradlabel, label distance=11pt]90:\tiny{$(-1,\!1,\!-\!1)$}}] at (sigma1=-1_sigma2=1_sigma3=-1_mu=0) {};
                \node[adgrad,label={[gradlabel]90:\tiny{$(1,\!1,\!1)$}}] at (sigma1=1_sigma2=1_sigma3=1_mu=0) {};
                \node[adgrad,label={[gradlabel, label distance=11pt]-90:\tiny{$(-\!1,\!-\!1,\!-\!1)$}}] at (sigma1=-1_sigma2=-1_sigma3=-1_mu=0) {};
            \end{tikzpicture}
            \caption{\(\mu = 0\)}
            \label{fig:grads-mu-zero}
        \end{subfigure}\begin{subfigure}[b]{\textwidth/3}
            \centering
            \tikzsetnextfilename{example-one-grads-mu-minus-one}
            \begin{tikzpicture}[scale=.74,]
                \drawgradients{-1}
                \draw[clarkegrads]
                (sigma1=-1_sigma2=1_sigma3=-1_mu=-1) node[bgrad, label={[gradlabel]90:\tiny{$(-\!1,\!1,\!-\!1)$}}] {} --
                (sigma1=-1_sigma2=1_sigma3=1_mu=-1) node[bgrad, label={[gradlabel]90:\tiny{$(-\!1,\!1,\!1)$}}] {} --
                (sigma1=1_sigma2=-1_sigma3=1_mu=-1) node[bgrad, label={[gradlabel]90:\tiny{$(1,\!-\!1,\!1)$}}] {} --
                (sigma1=1_sigma2=-1_sigma3=-1_mu=-1) node[bgrad, label={[gradlabel]90:\tiny{$(1,\!-\!1,\!-\!1)$}}] {} --
                cycle;
                \node[bgrad, label={[gradlabel]90:\tiny{$(1,\!1,\!-\!1)$}}] at (sigma1=1_sigma2=1_sigma3=-1_mu=-1) {};
                \node[bgrad, label={[gradlabel]-90:\tiny{$(-1,\!-\!1,\!1)$}}] at (sigma1=-1_sigma2=-1_sigma3=1_mu=-1) {};
                \node[adgrad, label={[gradlabel]90:\tiny{$(1,\!1,\!1)$}}] at (sigma1=1_sigma2=1_sigma3=1_mu=-1) {};
                \node[adgrad, label={[gradlabel]-90:\tiny{$(-\!1,\!-\!1,\!-\!1)$}}] at (sigma1=-1_sigma2=-1_sigma3=-1_mu=-1) {};
            \end{tikzpicture}
            \caption{\(\mu = -1\)}
            \label{fig:grads-mu-minus-one}
        \end{subfigure}
        \caption{Generalized derivatives of \(\varphi^\mu\) at \(\fixed{x}\) for different values of \(\mu\).
            The black dots indicate limiting gradients
            \(\bouligand{\varphi^\mu}(\fixed{x})\) and their convex hull, marked
            in gray, set of all Clarke gradients
            \(\clarke{\varphi^\mu}(\fixed{x})\).
            The red dots represent spurious gradients of smooth models
            \(\grad \varphi_{\sigma}^{\mu}(\fixed{x}) \notin
            \bouligand{\varphi}(\fixed{x})\).
        }
        \label{fig:grads-mu}
    \end{figure}
\end{example}
The example shows that MFKQ is not enough to ensure the conclusions of
Corollary~\ref{cor:bouligand}.
This is because MFCQ is only required for \(S_\sigma\) if \(\fixed{x}\in
\closure{S_\sigma}\).
LIKQ on the other hand is an algebraic condition, which by
Lemma~\ref{lem:full-rank-for-Sigma} imposes full rank conditions for all
\(\sigma \in \succmore{\fixed{\sigma}}\), independent of whether \(\fixed{x}
\in  \closure{S_\sigma}\).
These rank conditions are then used in Lemma~\ref{lem:all-essential} to proof
that \(\fixed{x} \in \closure{S_\sigma}\) indeed holds.

So far, it was shown that under LIKQ, definite choices for \(\partial
\abs{\cdot}(0)\), i.e., \(-1\) or \(1\), in the backward propagation yield
limiting gradients.
However, many AD tools use \(0\), see Table~\ref{tab:ad-tools}.
The next result justifies choices between \(-1\) and \(1\), again under the
assumption of LIKQ.
Such a choice is only valid if the respective
\(z_i[\fixed{x}]=0\), thus \(\sigma_i = 0\).
Thereby, the choice in the backward AD can only be made from \(\xi \in
\conv(\succmore{\fixed{\sigma}})\) and \(\grad \varphi_\xi(\fixed{x})\) is
computed.

\begin{theorem}\label{thm:ad-grads-are-clarke}
	For any \(\xi \in \conv(\succmore{\fixed{\sigma}})\) it holds
	\begin{equation}
        \label{eq:ad-grads-are-clarke}
		\sum_{\sigma \in \succmore{\fixed{\sigma}}} \beta_{\sigma, \xi} \grad \varphi_{\sigma}(\fixed{x})
        = \grad \varphi_{\xi}(\fixed{x})
        \text{ with }
        \beta_{\sigma, \xi} \coloneqq \prod_{i\in \fixed{\alpha}} \frac{\sigma_i \xi_i  + 1}{2}.
	\end{equation}
    In particular, this is a convex combination, and thus, under LIKQ, \(\grad
    \varphi_{\xi}(\fixed{x}) \in \clarke{\varphi}(\fixed{x})\).
\end{theorem}
\begin{proof}
    First observe that \(\beta_{\sigma, \xi} \geq 0 \) are convex coefficients
    because
	\begin{equation}
		\label{eq:prod_without_sigma}
		1
        = \prod_{i\in \fixed{\alpha}} 1
        = \prod_{i\in \fixed{\alpha}} \left(\frac{\xi_i +1}{2} + \frac{-\xi_i +1}{2}\right)
        = \sum_{\sigma \in \succmore{\fixed{\sigma}}} \prod_{i\in \fixed{\alpha}} \frac{\sigma_i \xi_i  + 1}{2}
        = \sum_{\sigma \in \succmore{\fixed{\sigma}}} \beta_{\sigma, \xi}.
	\end{equation}
    For \(k \notin \fixed{\alpha}\) it follows
    from~\eqref{eq:prod_without_sigma} and \(\xi_k = \fixed{\sigma}_k =
    \sigma_k\) for all \(\sigma \in \succmore{\fixed{\sigma}}\) that
	\begin{equation}
		\label{eq:prod_with_sigma}
		\sum_{\sigma \in \succmore{\fixed{\sigma}}} \sigma_k \beta_{\sigma, \xi} = \xi_k.
	\end{equation}
    For \(k \in \fixed{\alpha}\) the same holds true due to
	\begin{align*}
		\sum_{\sigma \in \succmore{\fixed{\sigma}}} \sigma_k \prod_{i\in \fixed{\alpha}} \frac{\sigma_i \xi_i  + 1}{2}
		&=
		\sum_{\sigma \in \succmore{\fixed{\sigma}}} \sigma_k \frac{\sigma_k\xi_k+1}{2}\prod_{\substack{i\in \fixed{\alpha} \\ i\neq k}} \frac{\sigma_i \xi_i  + 1}{2}
		\\
		&= \frac{\xi_k+1}{2} \sum_{\substack{\sigma \in \succmore{\fixed{\sigma}} \\ \sigma_k = 1}} \prod_{\substack{i\in 	\fixed{\alpha} \\ i\neq k}} \frac{\sigma_i \xi_i  + 1}{2}
		  +\frac{\xi_k-1}{2} \sum_{\substack{\sigma \in \succmore{\fixed{\sigma}} \\ \sigma_k = -1}} \prod_{\substack{i\in 	\fixed{\alpha} \\ i\neq k}} \frac{\sigma_i \xi_i  + 1}{2} \\
		&= \paren*{\frac{\xi_k+1}{2} + \frac{\xi_k-1}{2}} \prod_{\substack{i\in 	\fixed{\alpha} \\ i\neq k}} \left(\frac{\xi_i +1}{2} + \frac{-\xi_i +1}{2}\right)
= \xi_k.
	\end{align*}
    With \(\Sigma \coloneqq \diag(\sigma)\), \(\Xi \coloneqq \diag(\xi)\), now
    define
	\begin{equation}\label{eq:def-y-sigma-xi}
        \begin{aligned}
 		    y_\sigma &\coloneqq (I-M-L\Sigma)^{-T}(\Sigma b + d), \\
 		    y_{\xi} &\coloneqq (I-M-L\Xi)^{-T}(\Xi b + d),
        \end{aligned}
	\end{equation}
    whereby
    \(\grad \varphi_\sigma(\fixed{x}) = a + Z^T y_\sigma\)
    and
    \(\grad \varphi_\xi(\fixed{x}) = a + Z^T y_\xi\).
    Due to the strict lower triangular structure of \(M\) and \(L\) this means
	\begin{equation}\label{eq:y_k}
        \begin{aligned}
            y_{\sigma, k} &= \sigma_kb_k + d_k + \sum_{j = k+1}^{s} (M_{jk}y_{\sigma, j} + L_{jk}\sigma_ky_{\sigma, j}), \\
            y_{\xi, k} &= \xi_kb_k + d_k + \sum_{j = k+1}^{s} (M_{jk}y_{\xi, j} + L_{jk}\xi_ky_{\xi, j}),
        \end{aligned}
	\end{equation}
	for \(k = 1, \ldots, s\).
    In the following an induction over \(k=s, \ldots, 1\) shows that
	\begin{equation}
		\label{eq:y_identity_k}
		\sum_{\sigma \in \succmore{\fixed{\sigma}}}  \beta_{\sigma, \xi} y_{\sigma, k}  = y_{\xi, k}.
	\end{equation}
    For \(k=s\) the sum in~\eqref{eq:y_k} vanishes
    and~\eqref{eq:prod_without_sigma} and~\eqref{eq:prod_with_sigma} yield
	\begin{align*}
		\sum_{\sigma \in \succmore{\fixed{\sigma}}} \beta_{\sigma, \xi} y_{\sigma,s}
= b_s \sum_{\sigma \in \succmore{\fixed{\sigma}}}  \sigma_s   \beta_{\sigma, \xi}
		+
		d_s \sum_{\sigma \in \succmore{\fixed{\sigma}}} \beta_{\sigma, \xi}
		= \xi_s b_s + d_s
		= y_{\xi,s}.
	\end{align*}
    For the induction step assume that~\eqref{eq:y_identity_k} holds for all
    \(k > l \in \{1,\ldots, s-1\}\).
    Again using~\eqref{eq:y_k}~\eqref{eq:prod_without_sigma}
    and~\eqref{eq:prod_with_sigma} and the induction hypothesis yields
	\begin{align}
		\sum_{\sigma \in \succmore{\fixed{\sigma}}} \beta_{\sigma, \xi} y_{\sigma, l}
        &= \sum_{\sigma \in \succmore{\fixed{\sigma}}} \beta_{\sigma, \xi} \paren[\Bigg]{\sigma_l b_l + d_l + \sum_{j=l+1}^s  (M_{jl} y_{\sigma, j} + L_{jl} \sigma_l y_{\sigma, j})}\\
&= b_l \xi_l + d_l + \sum_{j=l+1}^s  \paren[\Bigg]{M_{jl} y_{\xi,j} + L_{jl} \sum_{\sigma \in \succmore{\fixed{\sigma}}} \beta_{\sigma, \xi} \sigma_l y_{\sigma, j}} \\
        &= y_{\xi, l} + \sum_{j=l+1}^s L_{jl} \paren[\Bigg]{\sum_{\sigma \in \succmore{\fixed{\sigma}}} \beta_{\sigma, \xi} \sigma_l y_{\sigma, j} - \xi_l y_{\xi, j}}.
	\end{align}
    It suffices to show that for \(k > l\)
    \begin{equation}
        \sum_{\sigma \in \succmore{\fixed{\sigma}}} \beta_{\sigma, \xi} \sigma_l y_{\sigma, k}
        = \xi_l y_{\xi, k}.
    \end{equation}
    For \(l \notin \fixed{\alpha}\) this follows immediately from the
    induction hypothesis~\eqref{eq:y_identity_k} because of \(\sigma_l =
    \fixed{\sigma}_l = \xi_l\) for all \(\sigma \in
    \succmore{\fixed{\sigma}}\).
    For \(l \in \fixed{\alpha}\), first observe that \(y_{\sigma, k} \) is
    independent of \(\sigma_l\), see~\eqref{eq:y_k}, and thus
    \begin{equation}
        \sum_{\substack{\sigma \in \succmore{\fixed{\sigma}}\\\sigma_l = 1}}
        \prod_{\substack{i\in \fixed{\alpha}\\ i \neq l}} \frac{\sigma_i \xi_i  + 1}{2} y_{\sigma, k}
        = \sum_{\substack{\sigma \in \succmore{\fixed{\sigma}}\\\sigma_l = -1}}
        \prod_{\substack{i\in \fixed{\alpha}\\ i \neq l}} \frac{\sigma_i \xi_i  + 1}{2} y_{\sigma, k}.
    \end{equation}
    Therefore, separating out the factor corresponding to \(i = l\) and
    splitting the sum over \(\succmore{\fixed{\sigma}}\) dependent on
    \(\sigma_l\) provides
	\begin{align}
		\sum_{\sigma \in \succmore{\fixed{\sigma}}} \beta_{\sigma, \xi} \sigma_l  y_{\sigma, k}
		&= \sum_{\sigma \in \succmore{\fixed{\sigma}}} \sigma_l\left(\frac{\sigma_l \xi_l + 1}{2}\right)  \prod_{\substack{i\in \fixed{\alpha}\\ i \neq l}} \frac{\sigma_i \xi_i  + 1}{2} y_{\sigma, k}  \\
		&= \sum_{\substack{\sigma \in \succmore{\fixed{\sigma}}\\\sigma_l = 1}} \left(\frac{\xi_l + 1}{2}\right) \prod_{\substack{i\in \fixed{\alpha}\\ i \neq l}} \frac{\sigma_i \xi_i  + 1}{2} y_{\sigma, k}
		 + \sum_{\substack{\sigma \in \succmore{\fixed{\sigma}}\\\sigma_l = -1}}\left(\frac{\xi_l -1}{2}\right)  \prod_{\substack{i\in \fixed{\alpha}\\ i \neq l}} \frac{\sigma_i \xi_i  + 1}{2} y_{\sigma, k} \\
		&= \sum_{\substack{\sigma \in \succmore{\fixed{\sigma}}\\\sigma_l = 1}} \left(\frac{\xi_l + 1}{2} + \frac{\xi_l -1}{2}\right) \prod_{\substack{i\in \fixed{\alpha}\\ i \neq l}} \frac{\sigma_i \xi_i  + 1}{2} y_{\sigma, k}.
	\end{align}
    Next, using the algebraic identity
    \begin{equation}
        \frac{\xi_l + 1}{2} + \frac{\xi_l - 1}{2}
        = \xi_l
        = \xi_l \paren*{\frac{\xi_l + 1}{2} + \frac{- \xi_l + 1}{2}},
    \end{equation}
    recombining the original sum over \(\succmore{\fixed{\sigma}}\),
    and finally applying the induction hypothesis yields
    \begin{align}
		\sum_{\sigma \in \succmore{\fixed{\sigma}}} \beta_{\sigma, \xi} \sigma_l  y_{\sigma, k}
		&= \xi_l \sum_{\substack{\sigma \in \succmore{\fixed{\sigma}}\\\sigma_l = 1}} \paren*{\frac{\xi_l + 1}{2} + \frac{- \xi_l + 1}{2}} \prod_{\substack{i\in \fixed{\alpha}\\ i \neq l}} \frac{\sigma_i \xi_i  + 1}{2} y_{\sigma, k} \\
&= \xi_l \sum_{\substack{\sigma \in \succmore{\fixed{\sigma}}\\\sigma_l = 1}}  \prod_{i\in \fixed{\alpha}} \frac{\sigma_i \xi_i + 1}{2} y_{\sigma, k}
         + \xi_l \sum_{\substack{\sigma \in \succmore{\fixed{\sigma}}\\\sigma_l = -1}} \prod_{i\in \fixed{\alpha}} \frac{\sigma_i \xi_i + 1}{2} y_{\sigma, k} \\
        &= \xi_l \sum_{\sigma \in \succmore{\fixed{\sigma}}} \beta_{\sigma, \xi} y_{\sigma,k}
        = \xi_l y_{\xi,k},
    \end{align}
    which concludes the induction.

    By writing \(\grad \varphi_{\xi}(\fixed{x})\) in terms of \(y_{\xi}\) and
    \(\grad \varphi_{\sigma}(\fixed{x})\) in terms of \(y_{\sigma}\), with the
    definitions in~\eqref{eq:def-y-sigma-xi}, the relation
    in~\eqref{eq:y_identity_k} finally shows
	\begin{align}
		 \grad \varphi_{\xi}(\fixed{x})
= a + Z^T y_{\xi}
		 = \sum_{\sigma \in \succmore{\fixed{\sigma}}} \beta_{\sigma, \xi} ( a + Z^T y_{\sigma, k})
= \sum_{\sigma \in \succmore{\fixed{\sigma}}} \beta_{\sigma, \xi} \grad \varphi_{\sigma}(\fixed{x}).
	\end{align}
    Under LIKQ Corollary~\ref{cor:bouligand} ensures \(\grad
    \varphi_{\sigma}(\fixed{x}) \in \bouligand{\varphi}(\fixed{x})\).
    Thereby, \(\grad \varphi_{\xi}(\fixed{x})\) is a convex combination of
    limiting gradients and thus a Clarke gradient.
\end{proof}

\section{Generalized Derivatives in Deep Learning}
\label{sec:nns}
Although having non-smooth training problems, ReLU activated (deep) feed
forward neural networks have attained huge popularity in recent years.
Since \(\relu(x) = \tfrac12(x + \abs{x})\), the non-smoothness stems from the
absolute value, hence said training problems are of class \(\Cabs\).
This section concerns the computation of generalized derivatives for training
problems of feed forward ReLU networks using AD.

Let \(J, N, M\in\nats\). For \(j=1, \ldots, J\) let \((u^{[j]},v^{[j]}) \in
\reals^N \times \reals^M\) be samples and corresponding labels.
Let \(T\in \nats\) be the number of hidden layers and \(N_t\in \nats\), \(t=1,
\ldots T\), the number of neurons in the \(t\)-th layer as well as
\(N_0\coloneqq N\) and \(N_{T+1}\coloneqq M\).

For \(t=1, \ldots, T+1\), the hidden layers and the output layer, define
weights \(W^{(t)} \in \reals^{N_t \times  N_{t-1}}\) as well as biases
\(b^{(t)} \in \reals^{N_t}\).
To evaluate the network function \(\operatorname{net}(u;
(\mathbf{W},\mathbf{b}))\) for a data point \(u^{(0)} \coloneqq u \in
\reals^N\) with
\begin{align}
	\mathbf{W} &\coloneqq (W^{(1)}, \ldots, W^{(T+1)}) \in \reals^{N_1 \times N_0}\times \cdots \times \reals^{N_{T+1} \times N_{T}} \eqqcolon \mathbb{W}, \\
	\mathbf{b} &\coloneqq (b^{(1)}, \ldots, b^{(T+1)}) \in \reals^{N_1}\times \cdots \times \reals^{N_{T+1}} \eqqcolon \mathbb{B}
\end{align}
compute for \(t=1, \ldots, T\),
\begin{equation}
	z^{(t)} \coloneqq W^{(t)}u^{(t-1)} + b^{(t)} , \quad
	u^{(t)} \coloneqq \relu( z^{(t)}  ) = \tfrac12( z^{(t)} + \abs{z^{(t)}}  ),
\end{equation}
and return \(u^{(T+1)} \coloneqq  h(W^{(T+1)}u^{(T)} + b^{(T+1)})\), where \(h
\colon \reals^{N_{T+1}} \to \reals^{M}\) is a smooth function, e.g., softmax.
The training problem, using a function \(\operatorname{loss}: \reals^M
\times \reals^M \to \reals\), is
\begin{equation}
	\min_{	\mathbf{W} \in \mathbb{W},	\mathbf{b} \in \mathbb{B} }
	\frac{1}{J}\sum_{j=1}^J \varphi^{[j]}(\mathbf{W},\mathbf{b}),
\end{equation}
where
\begin{equation}
	\varphi^{[j]}(\mathbf{W},\mathbf{b})
    \coloneqq \operatorname{loss}(\operatorname{net}(u^{[j]},(\mathbf{W},\mathbf{b}) ), v^{[j]}).
\end{equation}
Training problems in this form can be treated with a stochastic subgradient
method.
In the simplest case, a Clarke gradient of
\(\varphi^{[j]}(\mathbf{W},\mathbf{b})\) for a single \(j\in \set{1,\ldots, J}\)
is required for each update.
With auxiliary variables \(\mathbf{z}=(z^{(1)}, \ldots, z^{(T)}) \in
\reals^{N_1} \times \cdots \times \reals^{N_T}\), the function \(\varphi^{[j]}\)
can be formulated in the form of a \(\Cabs[2, s]\) function with \(s \coloneqq
N_1 + \ldots + N_T\) by the following definitions
\begin{align}
	f^{[j]}((\mathbf{W},\mathbf{b}),\abs{\mathbf{z}},\mathbf{z})
    &\coloneqq \operatorname{loss}( h(\tfrac{1}{2} W^{(T+1)}(z^{(T)}+ \abs{z^{(T)}}) + b^{(T+1)}) , v^{[j]}) \\[.4em]
	c^{[j]}((\mathbf{W},\mathbf{b}),\abs{\mathbf{z}},\mathbf{z})
    &\coloneqq
	\begin{pmatrix}
		W^{(1)}u^{[j]} + b^{(1)}  \\
		\tfrac12W^{(2)} (z^{(1)}+\abs{z^{(1)}}) + b^{(2)} \\
		\vdots \\
		\tfrac12W^{(T)} (z^{(T-1)}+\abs{z^{(T-1)}}) + b^{(T)}
	\end{pmatrix}.
\end{align}
This yields the following derivatives
\begin{equation}
	\label{eq:ZML_for_nn}
	Z^{[j]} =
	\begin{bmatrix}
		\dd{\mathbf{W}} c^{[j]}((\mathbf{W},\mathbf{b}),\abs{\mathbf{z}},\mathbf{z}) & I_s 	\end{bmatrix} ,
	\quad
	M^{[j]} = L^{[j]} =
	\frac12\begin{bmatrix}
		0& 0 &0 & \hdots & 0 \\
		W^{(2)} & 0 &0 & \hdots &0\\
		0 & W^{(3)} & 0 & \hdots &0 \\
		\vdots & \ddots & \ddots & \ddots & \vdots \\
		0 & \ldots & 0 & W^{(T)} & 0
	\end{bmatrix},
\end{equation}
where in the first term, the derivative should be with respect to the flattened
version of the weight matrix in order to obtain a matrix instead of a tensor.
Independent of this term and for all \((\mathbf{W},\mathbf{b})\) the considered
problem has LIKQ because \(Z^{[j]}\) has full row rank due to the identity
\(I_s\).
Thereby, the theory in the above sections applies and arbitrary choices for
\(\partial \abs{\cdot}(0)\) in the backward propagation yield Clarke
generalized gradients.

In practice, usually batch subgradients of
\begin{equation}
	\label{eq:batch_subgrad}
	\frac{1}{\abs{\mathcal{J}}}\sum_{j\in \mathcal{J}} \varphi^{[j]}(\mathbf{W},\mathbf{b}),
\end{equation}
with \(\mathcal{J} \subseteq \set{1, \ldots, J}\), are desired.
Unfortunately, LIKQ is in general not satisfied anymore in this setting, e.g.,
in the case where there are more active indices than trainable weights.
This means that choosing \(\partial \abs{\cdot}(0)\) arbitrary for every
absolute value and every sample in the batch does not yield a Clarke gradient.
In the next theorem it will be shown that for sums of terms sharing an
abs-normal form of particular structure, consistent choices yield generalized
gradients.
This structure applies to feed forward neural networks and these consistent
choices are naturally made in practical AD tools.
\begin{theorem}\label{thm:batch_nns}
    Let \(J,n, s\in \nats\), \(\mathcal{J} \subseteq \set{1, \ldots, J}\) and
    \(\fixed{x}\in \reals^n\).
    For \(j\in\mathcal{J}\), let \(\varphi^{[j]}\in \Cabs[2, s](\reals^n)\) be
    represented by an evaluation procedure \((f^{[j]}, c^{[j]})\) with solution
    operator \(z^{[j]}[\cdot]\), \(\fixed{z}^{[j]} \coloneqq
    z^{[j]}[\fixed{x}]\) and \(\fixed{\sigma}^{[j]} \coloneqq
    \sign(\fixed{z}^{[j]}) \in \set{-1,0,1}^s \).
    If there exists a common right-inverse \(Z^\pinv\) for all \(Z^{[j]}\), it
    holds for all \(\tau = \set{-1,1}^s\) that
	\begin{equation}
		\frac{1}{\abs{\mathcal{J}}} \sum_{j\in \mathcal{J}} \grad \varphi^{[j]}_{\sigma^{[j]}_\tau} (\fixed{x}) \in \bouligand{\paren[\Big]{\frac{1}{\abs{\mathcal{J}}}\sum_{j\in \mathcal{J}} \varphi^{[j]} }}(\fixed{x}) \text{ where }
		(\sigma^{[j]}_\tau)_k \coloneqq
		\begin{cases}
			\tau_k &\text{ if } \fixed{\sigma}_k^{[j]}  = 0\\
			\fixed{\sigma}_k^{[j]} &\text{ else.}
		\end{cases}
	\end{equation}
	Moreover, for \(\zeta \in [-1,1]^s\) it holds,
	\begin{equation}
		\frac{1}{\abs{\mathcal{J}}}\sum_{j\in \mathcal{J}} \grad \varphi^{[j]}_{\xi^{[j]}_\zeta} (\fixed{x}) \in \clarke{\paren[\Big]{\frac{1}{\abs{\mathcal{J}}}\sum_{j\in \mathcal{J}} \varphi^{[j]} }}(\fixed{x}) \text{ where }
		(\xi^{[j]}_\zeta)_k \coloneqq
		\begin{cases}
			\zeta_k &\text{ if } \fixed{\sigma}_k^{[j]}  = 0 \\
			\fixed{\sigma}_k^{[j]} &\text{ else.}
		\end{cases}
	\end{equation}
\end{theorem}
\begin{proof}
    Let \(\Sigma^{[j]}_\tau \coloneqq \diag(\sigma^{[j]}_\tau)\) and define \(v \in \reals^s\)
    inductively by \(v_1 \coloneqq \tau_1\) and
	\begin{equation}\label{eq:def-vk}
		v_k \coloneqq \tau_k\paren*{  1+\max_{j\in \set{1, \ldots, J}} \sum_{l=1}^{k-1} \abs{M_{kl}^{[j]} + L_{kl}^{[j]}(\sigma^{[j]}_\tau)_l} \abs{(I -M^{[j]} - L^{[j]}\Sigma^{[j]}_\tau)^{-1}v}_l }.
	\end{equation}
    Notice that this is well-defined because \(\paren{(I -M^{[j]} -
    L^{[j]}\Sigma^{[j]}_\tau)^{-1}v}_l\) only depends on the first \(l\)
    components of \(v\) due to the lower triangular structure of \(L^{[j]}\)
    and \(M^{[j]}\).
    Thereby, with \(d_\tau \coloneqq Z^\pinv v\) it holds for all \(j \in
    \mathcal{J}\) that
	\begin{align}
		\jac z_{\sigma^{[j]}_\tau}^{[j]}[\fixed{x}]d_\tau
		&= (I -M^{[j]} - L^{[j]}\Sigma^{[j]}_\tau)^{-1}Z^{[j]}Z^\pinv v \\
		&= (I -M^{[j]} - L^{[j]}\Sigma^{[j]}_\tau)^{-1}v.
	\end{align}
    By back substitution and~\eqref{eq:def-vk}, this yields for \(k=1, \ldots,
    s\) and all \(j\in \mathcal{J}\) that
    \begin{align}
		\tau_k (Dz_{\sigma^{[j]}}^{[j]}[\fixed{x}]d_\tau)_k 
		&= \tau_k v_k + \tau_k\sum_{l=1}^{k-1} \paren{M_{kl}^{[j]} + L_{kl}^{[j]}(\sigma^{[j]}_\tau)_l} \paren{(I -M^{[j]} - L^{[j]}\Sigma^{[j]}_\tau)^{-1}v}_l \\
		&\geq \tau_k v_k - \sum_{l=1}^{k-1} \abs{M_{kl}^{[j]} + L_{kl}^{[j]}(\sigma^{[j]}_\tau)_l} \abs{(I -M^{[j]} - L^{[j]}\Sigma^{[j]})^{-1}v}_l \\
		&\geq 1.
	\end{align}
    Hence, \(\sign((Dz_{\sigma^{[j]}_\tau}^{[j]}[\fixed{x}]d_\tau)_k) = \tau_k\)
    for all \(j\in \mathcal{J}\).
    Now define \(x_{\varepsilon, \tau}\coloneqq \fixed{x} + \varepsilon
    d_\tau\).
    By Taylor's theorem and Lemma~\ref{lem:consitent-solution} it holds for all
    \(j\in \mathcal{J}\) that
	\begin{equation}
		z^{[j]}[x_{\varepsilon, \tau}]
        = \fixed{z}^{[j]} + \varepsilon Dz_{\sigma^{[j]}_\tau}^{[j]}[\fixed{x}]d_\tau + \mathcal{O}(\varepsilon^2).
	\end{equation}
    Therefore, with the same arguments as in the proof of
    Lemma~\ref{lem:all-essential}, \(\sign(z^{[j]}[x_{\varepsilon, \tau}]) =
    \sigma^{[j]}_\tau \in  \set{-1,1}^s\), for all \(\varepsilon>0\) sufficiently
    small.
    Therefore, for all \(j \in \mathcal{J}\), it holds that \(\varphi^{[j]}\) is
    smooth at \(x_{\varepsilon, \tau}\) for \(\varepsilon > 0\) sufficiently
    small, hence
    \begin{equation}
        \grad \varphi^{[j]}(x_{\varepsilon, \tau})
        = \grad \varphi_{\sigma^{[j]}_\tau}^{[j]}(x_{\varepsilon, \tau}).
    \end{equation}
    Taking the sum over \(j\) and then the limit \(\varepsilon \downto 0\) yields
    the first desired identity.

    For the second identity fix \(j \in \mathcal{J}\) and note that
    \(\sigma^{[j]}_\tau \in \succmore{\fixed{\sigma}^{[j]}}\) for all \(\tau\in
    \set{-1,1}^s\). Now define the convex coefficients
	\begin{align*}
		\gamma_{\tau, \zeta} \coloneqq \prod_{i=1}^s \frac{\tau_i \zeta_i  + 1}{2}
	\end{align*}
    for all \(\tau \in \set{-1,1}^s\) and fixed \(\zeta \in [-1,1]^s\). By definition it holds \((\sigma^{[j]}_\tau)_i = \tau_i\) for all \(i\in \alpha^{[j]}\).
    Thereby, using that every $\tau \in \set{-1,1}^s$ has exactly one $ \sigma^{[j]} \in \succmore{\fixed{\sigma}^{[j]}}$ with $ \sigma_i^{[j]}=\tau_i=(\sigma^{[j]}_\tau)_i$  for all  $i\in \alpha^{[j]}$, it holds
    \begin{align}
    	\sum_{\tau \in \set{-1,1}^s}
        \gamma_{\tau, \zeta} \grad \varphi^{[j]}_{\sigma^{[j]}_\tau} (\fixed{x})
        \quad
    	&=
	    \sum_{\sigma^{[j]} \in \succmore{\fixed{\sigma}^{[j]}}}
	    \sum_{\substack{\tau \in \set{-1,1}^s\\ \sigma_i^{[j]} = \tau_i \text{ for } i\in \alpha^{[j]}}}
        \gamma_{\tau, \zeta}  \grad \varphi^{[j]}_{\sigma^{[j]}_\tau} (\fixed{x})\\
    	&=
	    \sum_{\sigma^{[j]} \in \succmore{\fixed{\sigma}^{[j]}}}
	    \sum_{\substack{\tau \in \set{-1,1}^s\\ \sigma_i^{[j]} = \tau_i \text{ for } i\in \alpha^{[j]}}}
        \gamma_{\tau, \zeta}  \grad \varphi^{[j]}_{\sigma^{[j]}} (\fixed{x}).
    \end{align}
    Altogether, using in addition Theorem~\ref{thm:ad-grads-are-clarke} with
    \(\beta_{\sigma, \xi}\) as defined there as well as \( (\xi^{[j]}_\zeta)_i = \zeta_i\) for all \(i\in \alpha^{[j]}\), it
    holds
	\begin{align}
		\sum_{\tau \in \set{-1,1}^s} \gamma_{\tau, \zeta} \grad \varphi^{[j]}_{\sigma^{[j]}_\tau} (\fixed{x})
		&= \sum_{\sigma^{[j]} \in \succmore{\fixed{\sigma}^{[j]}} }	\sum_{\substack{\tau \in \set{-1,1}^s\\\sigma_i^{[j]} = \tau_i \text{ for }   i\in \alpha^{[j]}    }} \prod_{i=1}^s \frac{\tau_i \zeta_i  + 1}{2} \grad \varphi^{[j]}_{\sigma^{[j]}} (\fixed{x}) \\
		&= \sum_{\sigma^{[j]} \in \succmore{\fixed{\sigma}^{[j]}} } \prod_{i\in \alpha^{[j]}} \frac{\sigma_i^{[j]} (\xi^{[j]}_\zeta)_i  + 1}{2}  \grad \varphi^{[j]}_{\sigma^{[j]}} (\fixed{x}) \!\!\!\!\!\! \sum_{\substack{\tau \in \set{-1,1}^s\\\sigma_i^{[j]} = \tau_i \text{ for }  i\in \alpha^{[j]}   }} \prod_{i\notin \alpha^{[j]}} \frac{\tau_i \zeta_i  + 1}{2} \\
		&= \sum_{\sigma^{[j]} \in \succmore{\fixed{\sigma}^{[j]}} } \beta_{\sigma^{[j]},\xi^{[j]}_\zeta} \grad \varphi^{[j]}_{\sigma^{[j]}} (\fixed{x}) = \grad \varphi^{[j]}_{\xi^{[j]}_\zeta} (\fixed{x}).
	\end{align}
	Averaging this identity over \(j\in \mathcal{J}\), yields
	\begin{align}
        \frac{1}{\abs{\mathcal{J}}}\sum_{j\in \mathcal{J}} \grad \varphi^{[j]}_{\xi^{[j]}_\zeta} (\fixed{x})
        \quad&=
        \sum_{\tau \in \set{-1,1}^s} \gamma_{\tau, \zeta} \frac{1}{\abs{\mathcal{J}}} \sum_{j\in \mathcal{J}} \grad \varphi^{[j]}_{\sigma^{[j]}_\tau} (\fixed{x}) \in \clarke{\paren[\Big]{	\frac{1}{\abs{\mathcal{J}}} \sum_{j\in \mathcal{J}} \varphi^{[j]} }}(\fixed{x}).
        \qedhere \end{align}
\end{proof}
In the above theorem, \(\tau \in \set{-1,1}^s\) and \(\zeta \in [-1,1]^s\) can
be seen as an a priori defined choice for \(\partial \abs{\cdot}(0)\) in the
case the respective absolute value is evaluated at zero.
Hence, the same a priori choice should be made for all samples.
In practice, this is usually \(\tau = (1, \ldots ,1)\) or \(\zeta = (0, \ldots
,0)\), see Table \ref{tab:ad-tools}.
It remains to apply Theorem~\ref{thm:batch_nns} to feed forward neural networks.

\begin{corollary}\label{cor:ad_works_nns}
    AD tools compute Clark generalized gradients in the training of feed forward
    neural networks by stochastic batch gradient descent.
\end{corollary}
\begin{proof}
    The abs-normal form of \(\varphi^{[j]}\) for training of neural networks has a
    pseudo-inverse for \(Z^{[j]}\) that is independent of \(j\) because by the
    computations in~\eqref{eq:ZML_for_nn},
	\begin{equation}
		Z^\pinv \coloneqq \begin{bmatrix}
			0 \\ I_s
		\end{bmatrix}
	\end{equation}
	is a suitable choice.
    Thereby, Theorem~\ref{thm:batch_nns} holds in this case.
    If in the backward propagation, a fixed choice \(\zeta \in [-1,1]^s\) is made
    for all samples, algorithmic differentiation on the batch
    problem~\eqref{eq:batch_subgrad} computes
	\begin{equation}
	    \frac{1}{\abs{\mathcal{J}}}\sum_{j\in \mathcal{J}} \grad \varphi^{[j]}_{\xi^{[j]}_\zeta} (\fixed{x}) \text{ where }
	    (\xi^{[j]})_k \coloneqq
	    \begin{cases}
	    	\zeta_k &\text{ if } \fixed{\sigma}_k^{[j]}  = 0 \\
	    	\fixed{\sigma}_k^{[j]} &\text{ else,}
        \end{cases}
    \end{equation}
    which is a Clark generalized gradient by Theorem~\ref{thm:batch_nns}.
\end{proof}
The existence of a shared pseudo-inverse relied on the identity block in
\(Z^{[j]}\).
This identity always arises if there is an independent bias added in every
ReLU activation function.
Therefore, Corollary~\ref{cor:ad_works_nns} applies for every neuron based
ReLU network with biases.

\section{Concluding remarks and further research}
It was shown that under the linear independence kink qualification for
abs-smooth problems generalized derivatives can be computed by naive
algorithmic differentiation.
In particular choosing any generalized gradient of the elementary functions
yields a generalized gradient of the composition.
This was shown for limiting gradients in Corollary~\ref{cor:bouligand} and for
Clarke generalized gradients in Theorem~\ref{thm:ad-grads-are-clarke}.
It can therefore be justified to use algorithmic differentiation as a
subgradient oracle in non-smooth optimization routines, if the objective
satisfies LIKQ at all iterates, which is a generic assumption.
Although LIKQ can not a priori be verified for deep neural networks,
algorithmic differentiation computes generalized gradients.
This is established in Theorem~\ref{thm:batch_nns} and
Corollary~\ref{cor:ad_works_nns} using similar techniques as in
Lemma~\ref{lem:all-essential}.

\printbibliography 

\end{document}
\typeout{get arXiv to do 4 passes: Label(s) may have changed. Rerun}